\documentclass[letterpaper, 10 pt, conference]{ieeeconf}

\IEEEoverridecommandlockouts
\usepackage{graphicx}

\usepackage{amsmath,amsthm}
\usepackage{cite}
\usepackage{amssymb}
\usepackage{times}
\usepackage{subfigure}
\usepackage{subfigmat}
\usepackage{epstopdf}
\usepackage{xcolor}
\usepackage{algorithm}
\usepackage{algorithmic}
\allowdisplaybreaks

\usepackage{multirow,multicol,threeparttable, pgfplots}
\usepackage{booktabs}
\usepackage{enumerate}
\usepackage{color}

\def\BibTeX{{\rm B\kern-.05em{\sc i\kern-.025em b}\kern-.08em T\kern-.1667em\lower.7ex\hbox{E}\kern-.125emX}}

\newtheorem{proposition}{Proposition}[section]

\renewcommand{\theenumi}{\textbf{Case} \arabic{enumi}}

\begin{document}

\title{\LARGE \bf Path Planning for Cooperative Routing of Air-Ground Vehicles}

\author{Satyanarayana G. Manyam$^{1}$, David W. Casbeer$^{2}$, and Kaarthik Sundar$^{3}$%
\thanks{$^{1}$National Research Council Fellow, Air Force Research Laboratory, Dayton-OH, 45433,  \texttt{msngupta@gmail.com}, }%
 \thanks{$^{2}$Research Scientist, Control Science Center of Excellence, Air Force Research Laboratory, WPAFB, OH, 45433, \texttt{david.casbeer@us.af.mil},}%
 \thanks{$^{3}$Graduate Student, Dept. of Mechanical Engineering, Texas A\&M University, College Station, TX 77843.}
}

\maketitle
\thispagestyle{empty}
\pagestyle{empty}

\begin{abstract}
We consider a cooperative vehicle routing problem for surveillance and reconnaissance missions with communication constraints between the vehicles. We propose a framework which involves a ground vehicle and an aerial vehicle; the vehicles travel cooperatively satisfying the communication limits, and visit a set of targets. We present a mixed integer linear programming (MILP) formulation and develop a branch-and-cut algorithm to solve the path planning problem for the ground and air vehicles. The effectiveness of the proposed approach is corroborated through extensive computational experiments on several randomly generated instances.
\end{abstract}

\section{Introduction} \label{sec:intro}

In this paper, we present a path planning problem involving an Unmanned Aerial Vehicle (UAV) and a ground vehicle for intelligence, surveillance and reconnaissance (ISR) missions. UAVs are being routinely used in military applications such as border patrol, reconnaissance, and surveillance expeditions, and civilian applications \cite{Frew2009, Curry2004, ZajkowskiT2006}. They are prime candidates for ISR missions due to their several advantages such as portability and low risk, to name a few. A typical ISR mission would require the UAVs to collect images, videos, or sensor data and transmit them to a ground/base station. The data collected in these ISR missions are most often very time sensitive and ideally, the data would be useful only if it is processed in real-time or near real-time. We propose a framework to meet this objective, and solve the underlying problem of planning paths for the UAV and the ground vehicle.

Cooperative control and path planning for a team of vehicles (UAVs or UAVs together with ground vehicles) has been a problem of interest and has received wide attention during the past decade (see \cite{rasmussenbook, kumarram, tsourdosbook, gilsparks, lasfarg, smithhdp, sujithicuas}). The problem of communication-constrained vehicle routing for a team of vehicles has been addressed in \cite{mostofidynamic, hsiehcowley, basuredi, Sabo:2014aa}. More specifically, authors in \cite{mostofidynamic} study the problem of deploying a team of mobile agents to periodically monitor several points of interest. Authors in \cite{hsiehcowley} perform an experimental study of the strategies to maintain end-to-end communication links for a team of robots in reconnaissance missions.  Heuristic approaches for a variation of the traditional vehicle routing problems with a simple communication model has been studied in \cite{Sabo:2014aa}. In \cite{sujithicuas}, the authors considered a forest mapping application involving an UAV and a ground vehicle; they develop algorithms to plan  routes for the fuel constrained UAV and the ground vehicle, which can act as refueling station for the UAV during the mission. This is the first work in the literature that proposes an approach using one UAV and one ground vehicle.

In this article, we propose a reconnaissance and data collection methodology using a UAV and a ground vehicle. The two vehicles are subject to communication constraints and they work cooperatively to accomplish the mission. The choice of using a ground vehicle and a UAV for ISR missions is motivated by the lack of roads and the presence of geographical obstacles such as rivers, lakes, mountains, etc. that must be circumvented to reach target locations. In the scenarios where ground vehicle may not be able to reach certain targets, a UAV could be used to gather information.  However, due to size, weight, and power restrictions the UAV would be unable to transmit the data from remote targets to the ground station. Furthermore, in ideal situations, the ground vehicle would have a communication link to the ground station, perhaps through satellite or by other means. Hence, in this paper, the UAV is required to stay connected to the ground vehicle, under the thought that data can be instantaneously transmitted to the ground station. We also do not impose any specific constraint on communication for the ground vehicle, and thus in cases where communication is degraded and the ground vehicle cannot maintain contact with the ground station, the cooperative routing problem here will still yield a viable solution.  In this case, one could think of this problem as a cooperative routing problem to minimize the effort taken by the vehicles to acquire the data and then ferry it back to the ground station. The problem is formally defined in the following section.

\begin{figure}[htpb]
\centering
\includegraphics[height=2in]{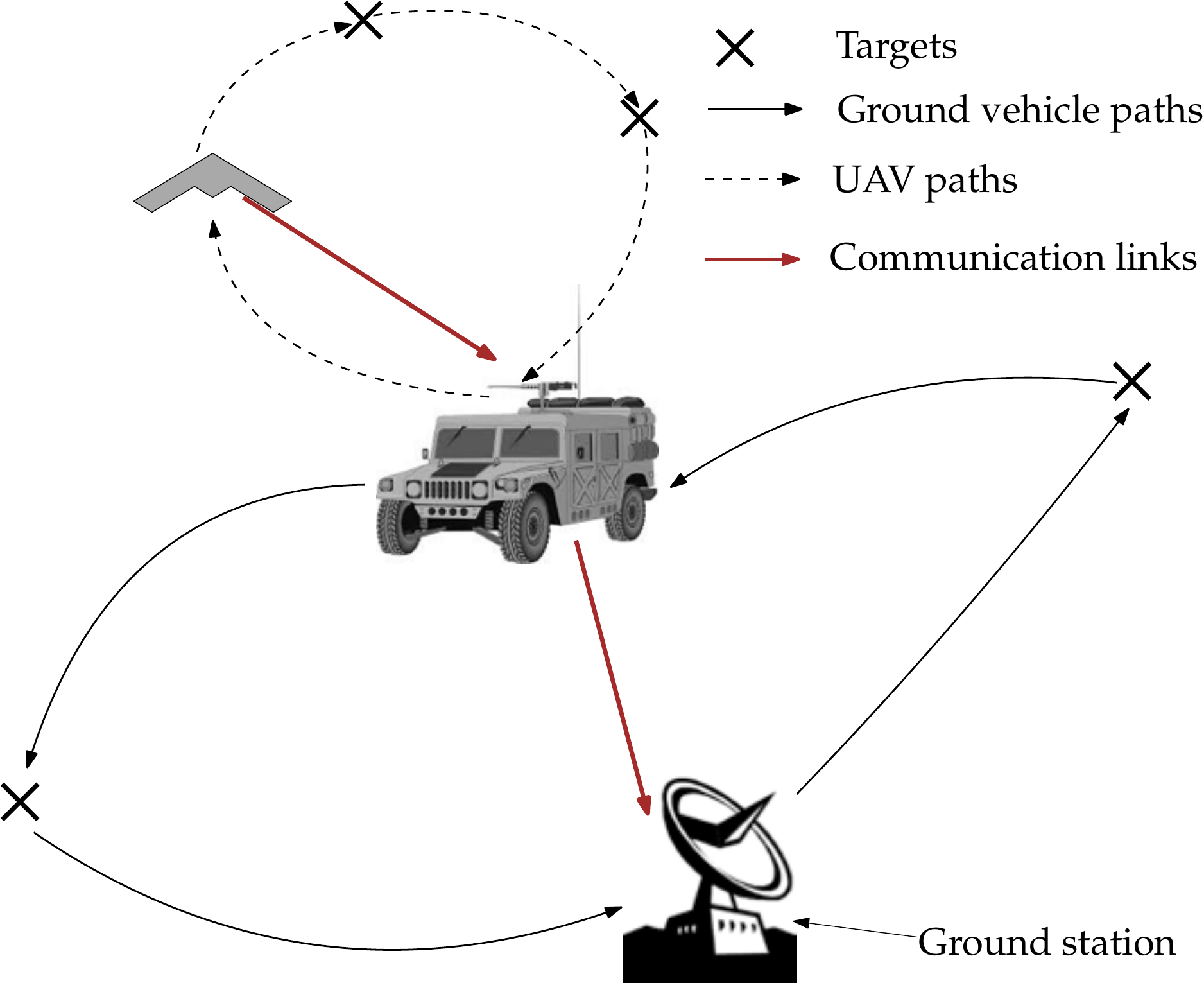}
\caption{Air vehicle - Ground vehicles paths}
\label{fig:agvroute}
\end{figure}

\subsection{Path planning problem}
The path planning problem considered in this article is formally defined as follows: we are given the locations of a set of targets and a ground station where the ground vehicle is initially stationed. The UAV is being carried by the ground vehicle. For each vehicle, we are given a travel cost between every pair of targets. This cost could be the Euclidean distance between the pair of targets or could also depend on the terrain.  Given this, the objective of the problem is to determine paths of minimum cost for the ground vehicle and the UAV such that following conditions are satisfied:
\begin{itemize}
\item every target is visited either by the ground vehicle or by the UAV, and
\item the UAV can always establish a reliable communication link with the ground vehicle.
\end{itemize}
The sequence of actions for a feasible mission is as follows. The ground vehicle starts at the ground station and visits the first target in its route, then the UAV is deployed from the first target and it collects the data from a subset of targets according to its plan and returns to the ground vehicle. Then the ground vehicle proceeds to the next stop. This process is repeated until the data is collected from all of the targets.  An illustration of a typical route is shown in Fig. \ref{fig:agvroute}. Once all the targets are visited, the ground vehicle carrying the UAV returns to the ground station. More specifically, the objective of determining the routes for the vehicles involves (i) identifying the stops and order to visit them for the ground vehicle and (ii) at each stop, identifying the subset of targets and order in which the UAV has to visit them. Furthermore, to ensure that the UAV can always establish a communication link with the ground vehicle, we enforce the constraint that the UAV has to stay within a distance of $R$ units from the ground vehicle at all times during the mission. We call this problem the \emph{cooperative air-ground vehicle routing problem} (CAGVRP). A feasible solution to the problem is shown in Fig. \ref{fig:sampletour}.

\begin{figure}[htpb]
\centering{}
\includegraphics[width=2.5in]{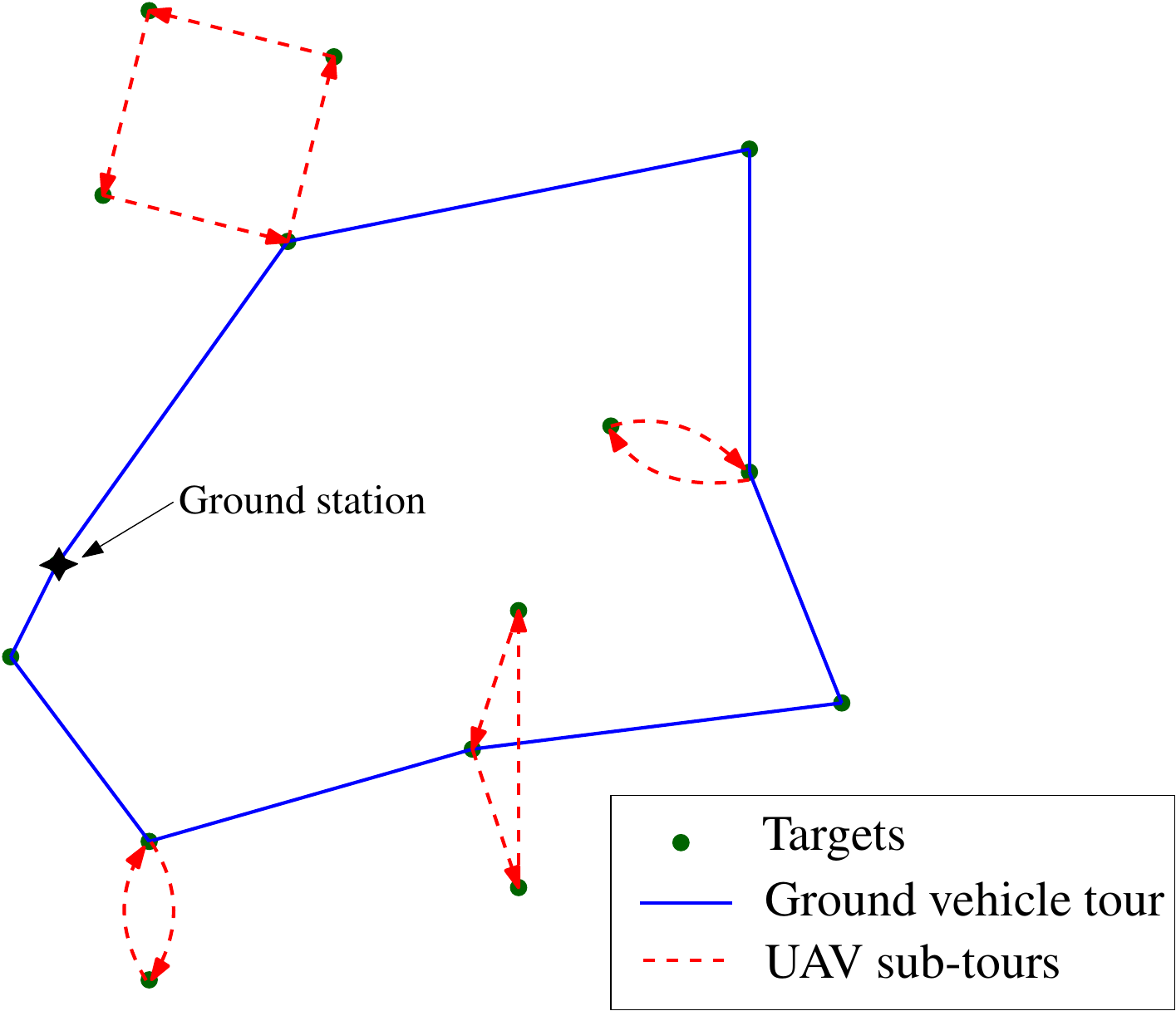}
\caption{A sample tour of the CAGVRP}
\label{fig:sampletour}
\end{figure}

\subsection {Related Literature}
The CAGVRP is NP-hard because it is a generalization of the traveling salesman problem. Authors in \cite{smithhdp} propose a framework similar to the CAGVRP involving a truck, traveling on a road network and a quadrotor for a package delivery problem in urban environments. This problem differs from the CAGVRP in two aspects: (i) the ground vehicle or the truck is restricted to travel along a road network and (ii) the quadrotor has to return to the ground vehicle after each delivery. The authors present a transformation algorithm to transform the problem to a generalized traveling salesman problem, whereas the focus of this paper is to develop algorithms to obtain an optimal solution to the CAGVRP. CAGVRP also resembles the two echelon vehicle routing problem \cite{perbolitwoech}; the key difference is in CAGVRP the locations at which the UAV tours originate are not known, which makes it even harder to solve.

Other problems addressed in the literature that are similar to the CAGVRP are the ring-star problem \cite{laportersp, kaarthikmrsp, baldaccimrsp, baldaccicaprsp} and the hierarchical ring-network problem \cite{golddam, altinkcor, carrollinoc, klincreview, shiself, stidsenringnet}. The ring-star problem aims to find a minimum cost cycle (or ring) through a subset of targets, and the targets that do not lie on the cycle should be assigned to one of the targets in the cycle. Algorithms to obtain an optimal solution based on the branch-and-cut paradigm are presented for the ring-star problem in \cite{laportersp}, for its multiple-depot variant in \cite{kaarthikmrsp}, and for the capacitated variants in \cite{baldaccimrsp, baldaccicaprsp}. In CAGVRP, not only must the targets that does not lie on the main ring be assigned to the targets on the ring, but there must be sub-tours chosen for all the off-ring targets. 

Another closely related problem is the hierarchical ring-network problem (HRNP) which aims to find a hierarchical two-layer ring network. The top layer consists of a federal ring (analogous to the ground vehicle tour) which establishes connection between a number of node-disjoint metro rings (analogous to the UAV sub-tours) in the bottom layer. Authors in \cite{altinkcor} present heuristics and an approximation algorithm to solve the HRNP; they assume that the number of metro rings and the nodes at which the metro rings are attached to the federal ring are given. In \cite{shiself}, heuristics and enumeration methods are given to solve the HRNP, where a certain demand must be satisfied between every pair of targets. Heuristics to solve a variant of the HRNP are presented in \cite{golddam}. 

Another variant of HRNP, where the bottom layer could be rings or ring-stars is solved in \cite{carrollinoc}. The authors assume the number of possible local rings are given. The authors in \cite{stidsenringnet} use branch-and-price algorithm to solve the HRN problem with demands to be satisfied between every pair of targets. {Their solution procedure involved two steps: in the first step, they solve a modified HRNP which ignores the cost of the federal ring and design the metro rings; in the second step, they solve a generalized traveling salesman problem on the metro rings.} In all the aforementioned papers, the cost of the federal ring is either ignored or assumed to be equal to that of the metro ring. The CAGVRP differs from the HRNP in the following aspects: (i) the travel costs for the ground vehicle and the UAV are different, (ii) we have an additional distance constraint \emph{i.e.,} the distance between the UAV and the ground vehicle has to be within $R$ units throughout the mission, and (iii) we attempt to solve the coupled problem of finding paths for the ground vehicle and the UAV to minimizes the total travel cost.

This article is organized as follows: We present the problem formulated as mixed-integer linear program (MILP) in Section \ref{sec:prbfrm}. In Section \ref{sec:bandc}, we explain the branch-and-cut framework used to solve the MILP formulation. Computational results are reported in Section \ref{sec:compres} and conclusions are made in Section \ref{sec:conc}.

\section{Problem Formulation} \label{sec:prbfrm}

This section presents a mixed integer linear programming formulation for the CAGVRP. Let $T$ denote the set of targets $\{t_0,\dots,t_n\}$; $t_0$ is the ground station. The CAGVRP is defined on a mixed graph $G=(T, E\cup A)$, where $E$ and $A$ are a set of undirected edges and directed arcs, respectively, joining any two targets in $T$. Each edge $e = (i,j) \in E$ is associated with a non-negative cost $c_e$ required for the ground vehicle to traverse the edge $e$ (if $e$ connects the vertices $i$ and $j$, then $(i,j)$ and $e$ will be used interchangeably to denote the same edge in $E$). Similarly, each arc $[i,j] \in A$ is associated with a non-negative cost $d_{ij}$ required for the UAV to travel from target $i$ to target $j$. We also associate with each arc $[i,j]\in A$, an auxiliary cost $f_{ij}$ to enforce the mission constraint that the UAV should always be within a distance of $R$ units from the ground vehicle. For any $i,j \in T$, $f_{ij}$ is set to $0$ if the Euclidean distance between targets $i$ and $j$ is less than $R$ units, and is set to an arbitrarily large value otherwise. 

We associate vectors $\mathbf x \in \mathbb{R}^{|E|}$, $\mathbf w \in \mathbb{R}^{|A|}$ and $\mathbf y \in \mathbb{R}^{|A|}$ with each feasible solution $\mathcal F$. The component $x_e$ of $\mathbf x$, associated with edge $e \in E$, is a binary variable which takes a value $1$ if the edge $e$ is traversed by the ground vehicle, and $0$ otherwise. Each component $w_{ij}$ of $\mathbf w$, associated with the directed arc $[i,j] \in A$  is a binary variable; it denotes if the arc is present in a UAV sub-tour. Similarly, $y_{ij}$, a component of $\mathbf y$ is a binary assignment variable that takes a value $1$ if the target $i$ is assigned to the UAV sub-tour that originates from the target $j$, and $0$ otherwise. Note that if a target $i$ is visited by a ground vehicle then the assignment variable $y_{ii}$ is equal to $1$ (the target is assigned to itself). 

For any $S\subset T$, we define $\gamma(S) = \{(i,j) \in E: i,j \in S\}$, $\delta(S) = \{(i,j)\in E: i \in S, j\notin S\}$, $\delta^+(S) = \{[i,j]\in A: i\in S, \,j\notin S\}$, and $\delta^-(S) = \{[j,i]\in A: i\in S, j\notin S\}$. If $S = \{i\}$, we simply write $\delta(i)$, $\delta^+(i)$ and $\delta^-(i)$ instead of $\delta(\{i\})$, $\delta^+(\{i\})$ and $\delta^-(\{i\})$ respectively. A mixed-integer program formulated using the above variables for the CAGVRP is as follows:
\begin{flalign}
&(\mathcal{F}_1) \quad\mbox{Minimize}  \sum_{e \in E} c_e x_e + \sum_{(i,j) \in A} d_{ij} w_{ij}  + \sum_{(i,j) \in A} f_{ij} y_{ij} \nonumber \\ 
&\mbox{Subject to} \quad \quad \nonumber \\
&\sum_{e \in \delta(i)} x_e  = 2y_{ii} \qquad\forall i \in T, \label{eq:degx}\\
&\sum_{j \in T} w_{ij} =1, \qquad \forall i \in T, \label{eq:outdegw}\\
&\sum_{i \in T} w_{ij} =1, \qquad \forall j \in T, \label{eq:indegw}\\
&w_{ij}  \le \sum_{k \in T} y_{ik} y_{jk}, \qquad \forall [i,j] \in A, \label{eq:wij}\\
&\sum_{e \in \delta (S)}x_e \ge 2 \sum_{j \in S}y_{ij}, \qquad \forall i \in S, \, t_0 \notin S, \, S \subset T, \label{eq:xsec}\\
&\sum_{[i,j] \in \delta^+(S)} w_{ij} \ge 1-\sum_{j \in S} y_{ij}, \qquad \forall i \in S, \, S \subset T, \label{eq:wsec1}\\
&\sum_{[i,j] \in \delta^-(S)} w_{ij} \ge 1-\sum_{j \in S} y_{ij}, \qquad \forall i \in S, \, S \subset T, \label{eq:wsec2} \\
&x_e \in \{0,1\}, \qquad \forall e \in E, \text{ and } \label{eq:xint} \\
&y_{ij}, w_{ij} \in \{0,1\} \qquad \forall [i,j] \in A. \label{eq:ywint} 
\end{flalign}
\begin{figure}
\centering{}
\includegraphics[width=2in]{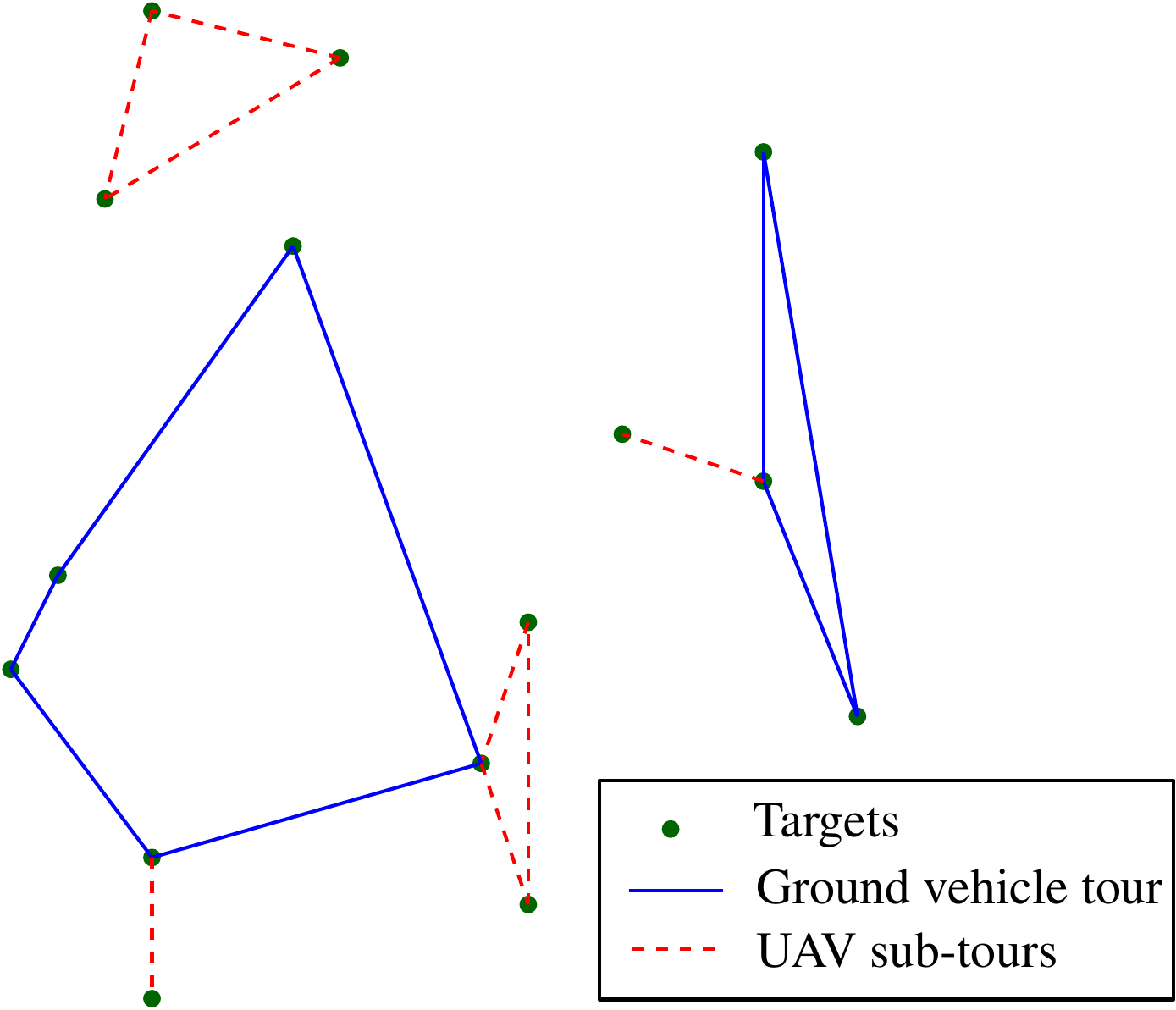}
\caption{Infeasible solution to the CAGVRP because the ground vehicle tour consists of two disjoint sub-tours and there exists a separated UAV sub-tour.}
\label{fig:infeas}
\end{figure}

The objective of the problem is to minimize the total cost of travel for the ground vehicle and the UAV. In the above formulation, the constraints in \eqref{eq:degx} ensure that the number of undirected ground vehicle edges incident on any target $i\in T$ is equal to $2$ if and only if the target $i$ is assigned to itself ($y_{ii} = 1$). The UAV sub-tours connectivity constraints are enforced using the equations \eqref{eq:outdegw} and \eqref{eq:indegw}. For each target, these constraints ensure that there are two UAV sub-tour edges incident on the target, one incoming and one outgoing edge. This is in contrast to the ground vehicle routes. The UAV has to visit every target $i\in T$ irrespective of the value of $y_{ii}$. When the value of the $y_{ii}=1$, the constraints in \eqref{eq:outdegw} and \eqref{eq:indegw} would result in a trivial UAV sub-tour \emph{i.e.,} $w_{ii} = 1$. The constraints in Eq. \eqref{eq:wij} ensure that a sub-tour edge of the UAV can exist between two targets only if those two targets are assigned to the same target. One can observe that these constraints are non-linear. We also note that the degree constraints in Eq. \eqref{eq:degx} are not sufficient to formulate the route for the ground vehicle. 

With just the degree constraints and the assignment constraints, the minimization problem may return a  solution with separated tours for the ground vehicle and/or isolated sub-tours for the UAV. An instance of such an infeasible solution is shown in Fig. \ref{fig:infeas}. In general, there are three different ways to eliminate these isolated tours/sub-tours. This is done by introducing a new set of constraints \textit{viz.} \textit{MTZ} constraints or flow constraints or sub-tour elimination constraints (SEC) \cite{applegatebook}. For vehicle routing problems, the cut based sub-tour elimination (SEC) constraints are known to outperform \textit{MTZ} and flow constraints  computationally\cite{tothvigo}. Therefore we use SEC constraints \eqref{eq:xsec}, \eqref{eq:wsec1} and \eqref{eq:wsec2} to eliminate such solutions containing separated tours for the ground vehicle and for the UAV tours. These constraints ensure that for any sub-set of targets $S$, such that all the targets in $S$ are assigned to a target not in $S$, then a UAV sub-tour containing only the targets in $S$ cannot exist. The constraints \eqref{eq:wsec1} and \eqref{eq:wsec2} mandate that there should be an arc coming into the set $S$ and an arc going out of the set $S$. The novelty of this model is it addresses the  connectivity of the tours without \textit{MTZ} or flow constraints, rather this model uses SEC constraints to ensure a connected tours. The exponential number of SEC constraints are not explicitly enumerated here but are addressed in the branch-and-cut framework only when they are violated.

As mentioned previously, the formulation $\mathcal F_1$ is non-linear due to the constraints in \eqref{eq:wij}. These constraints can be linearized by introducing an auxiliary variable $z_{ijk}$ for every triplet of targets $i,j,k\in T$. 

\begin{proposition}
The constraints in \eqref{eq:wij} can be reformulated using the following constraints: 
\begin{flalign}
&w_{ij} \le \sum_{k \in T} z_{ijk}, \qquad \forall [i,j] \in A, \label{eq:wijlin}\\
&z_{ijk} \le y_{ik} \qquad \forall i, j, k \in T, \label{eq:zlin1}\\
&z_{ijk} \le y_{jk} \qquad \forall i, j, k \in T, \text{ and}\label{eq:zlin2} \\
&z_{ijk} \ge y_{ik} + y_{jk} -1 \qquad \forall i, j, k \in T. \label{eq:zlin3} 
\end{flalign}
\end{proposition}
\begin{proof}
We prove the proposition by considering the following two cases (i) $y_{ik} = y_{jk} = 1$ and (ii) $y_{ik} = 0$ or $y_{jk} = 0$. For case (i), observe that $z_{ijk}$ will take a value of $1$ due to the constraints \eqref{eq:zlin1}--\eqref{eq:zlin3}. Similarly for case (ii), suppose $y_{ik} = 1$ and $y_{jk} = 0$, then $z_{ijk} = 0$. A similar argument holds when $y_{jk} = 1$ and $y_{ik} = 0$. 
\end{proof}
Now, the mixed integer linear programming formulation for the CAGVRP is given by
\begin{flalign}
&(\mathcal{F}_2) \quad\mbox{Minimize}  \sum_{e \in E} c_e x_e + \sum_{(i,j) \in A} d_{ij} w_{ij}  + \sum_{(i,j) \in A} f_{ij} y_{ij} \nonumber \\ 
&\mbox{Subject to} \,\, \text{ \eqref{eq:degx}--\eqref{eq:indegw}, \eqref{eq:wijlin}--\eqref{eq:zlin3}, and \eqref{eq:xsec}--\eqref{eq:ywint}}. \nonumber 
\end{flalign}
If the integrality restrictions in the constraints \eqref{eq:xint} and \eqref{eq:ywint} are relaxed, then we call that model a linear programing relaxation of the original MILP. In the following subsection, we shall strengthen the linear programming relaxation of the formulation $\mathcal F_2$ by introducing additional valid inequalities (a constraint is called a valid inequality if it does not remove any feasible solution).

\subsection{2-matching inequalities}
This section introduces a class of valid inequalities for the CAGVRP. These inequalities are derived from the \emph{2-matching} inequalities for the traveling salesman problem \cite{grotpadpoly, grotschelhollandmp, fischettigtsp} and is also valid for the ring-star problem and its variants \cite{laportersp,kaarthikmrsp, sundar2015exact}. Specifically, we consider the following inequality: 
\begin{flalign}
\sum_{e \in \gamma(H)}x_e + \sum_{e \in I}x_e \le &\sum_{i \in H} y_{ii} + \frac{|I|-1}{2}, \label{eq:twomatch} 
\end{flalign}
for all $H \subseteq T$ and $I \subseteq \delta(H)$. Here $H$ is called the handle and $I$, the teeth. $H$ and $I$ satisfy the following conditions: (i) no two edges in the teeth are incident on the same target and (ii) the number of edges in the teeth is odd and greater than equal to $3$. The proof of validity of the above inequality is given by the following proposition. One can refer to \cite{laportersp, kaarthikmrsp, sundar2015exact} for the proof of this proposition.

\begin{proposition} 
The 2-matching inequalities in Eq. \eqref{eq:twomatch} is valid of any feasible solution to the CAGVRP.
\end{proposition}
The constraints in \eqref{eq:twomatch} are equivalent to the \emph{blossom's inequality} \cite{Edmonds1965} for the 2-matching problem and a special case of the comb inequalities for the symmetric traveling salesman problem. 

\section{Branch-and-cut algorithm} \label{sec:bandc}
In this section, we outline the main components of the branch-and-cut algorithm. As mentioned previously, the branch-and-cut algorithm is an intelligent enumeration technique to find an optimal solution to an MILP. Let $\bar{\tau}$ denote the optimal solution for a problem instance. 


\begin{enumerate}[1.]
\item \textit{Initialization:} Set the iteration count to $k \gets 1$ and the initial upper bound $\bar{\alpha}$ on the optimal objective value as $+\infty$. The initial linear programming subproblem is then defined by the constraints \eqref{eq:degx}--\eqref{eq:indegw}, \eqref{eq:wijlin}--\eqref{eq:zlin3}, and $0 \leq x_e \leq 1$, $0 \leq y_{ij},w_{ij} \leq 1$ for every $e \in E$ and $[i,j] \in A$ respectively. This initial linear programming subproblem is solved and inserted into a list $\mathcal L$. 
\item \textit{Termination check and subproblem selection:} If the list $\mathcal L$ is empty, then stop. Otherwise, select a subproblem from the list with the lowest objective value; this choice of subproblem is called the best-first policy \cite{wolseybook}. \label{step2}
\item \textit{Subproblem solution:} Set $k \gets k+1$. Let $\alpha$ be the solution objective value. If $\alpha \geq \bar{\alpha}$, then go to step \ref{step2}. Otherwise, if the solution is feasible for the CAGVRP, set $\bar{\alpha} \gets \alpha$, update $\bar{\tau}$ and go to step \ref{step2}; if the solution is infeasible, then go to step \ref{step5}. \label{step3}
\item \textit{Constraint separation and generation:} Introduce violated sub-tour elimination constraints \eqref{eq:xsec}, connectivity constraints \eqref{eq:wsec1} and \eqref{eq:wsec2} and 2-matching constraints \eqref{eq:twomatch}. If no constraints can be generated using the current fractional solution, then go to \ref{step6}, else go to step \ref{step3}. \label{step5}
\item \textit{Branching:} Create two subproblems by branching on a fractional variable and insert both the subproblems in the list $\mathcal L$. \label{step6}
\end{enumerate}

At the end of the branch-and-cut algorithm $\bar{\tau}$ will have the optimal solution. In the following sections, we detail the constraint separation procedure used in the step \ref{step5}. This procedure finds violated sub-tour elimination constraints, connectivity constraints and 2-matching constraints from a fractional solution to the subproblems defined in the algorithm, if any.

\subsection{Separation of connectivity constraints in \eqref{eq:xsec}--\eqref{eq:wsec2}}
In this section, we discuss exact separation procedures for separating out the constraints in \eqref{eq:xsec}--\eqref{eq:wsec2} given a fractional solution. To this end, let $(x^*, y^*, w^*)$ denote a fractional solution. We first construct a support graph ($G^*$) based on the fractional solution; $G^*=(V^*, E^*)$, $V^* = \{i \in T: 0 < y^*_{ii} < 1 \}$ and $E^* = \{e \in E: 0 < x^*_e < 1 \}$. Then, we check if the graph $G^*$ is connected; if it is not connected, each vertex set $S$ corresponding to an individual connected component, such that $t_0 \notin S$, generates a violated constraint (\ref{eq:xsec}) for each $i \in S$. If $G^*$ is connected, we find the subset of nodes $S$ with minimum value of $\sum_{e \in \delta(S)} x^*_e$. This is done by solving a problem of computing the max-flow on $G^*$, where the capacity of each edge $e$ is set to $x^*_e$. We compute the minimum cut of all pairs of nodes on $G^*$ and consider the node partition $S$ that does not contain $t_0$. We check if this set violates  constraint \eqref{eq:xsec} for each $i \in S$. If the constraint is violated for any $i$, we add the corresponding inequality to the existing pool of inequalities. 

We use a similar procedure to find violated connectivity constraints \eqref{eq:wsec1}--\eqref{eq:wsec2} for the UAV sub-tours. Based on the fractional solution $w^*$, we construct graph $G^*_u = (V^*, A^*)$, $A =\{ [i, j]: 0<w_{ij}<1 \} $. Similar to the previous procedure we find the violated constraints \eqref{eq:wsec1}--\eqref{eq:wsec2} defined by $S$, such that the subset $S$ does not contain any targets that are visited by the ground vehicle.

\subsection{$2$-matching Constraints}
To find the violated $2$-matching constraints, we follow the procedure described in \cite{laportersp, fischettigtsp}. We consider each connected component $H$ of $G^*$ as a handle of a possible violated 2-matching inequality, whose two-target teeth correspond to edges in $e\in \delta(H)$ with $x_e^* = 1$. We reject the inequality if the number of teeth is even. If the inequality is violated, then we add this inequality to the pool of violated inequalities. 

\section{Computational Results} \label{sec:compres}

In this section, we discuss the computational results of the branch-and-cut algorithm. The algorithm was implemented in Julia \cite{Julia} using JuMP \cite{Jump} (a mathematical modelling framework for the Julia programming language) and CPLEX $12.6$. The internal CPLEX cut-generation routines were disabled and CPLEX was used only to manage the enumeration (branch-and-bound) tree. All the simulations were performed on a Macbook Pro with an Intel Core i5, 2.7 GHz processor. The performance of the algorithm was tested on several randomly generated instances. 

\emph{Instance generation:} We generated twenty instances, five for each value of $|T| \in \{10, 20, 30, 40\}$. The coordinates of the targets are generated randomly from a uniform distribution in a $100 \times 100$ grid. For every pair of targets $i,j$, the travel cost for the ground vehicle to traverse the edge $(i,j)$ is chosen to be equal to the Euclidean distance between them ($l_{ij}$), and the cost of travel by the UAV is chosen to be $\alpha \cdot l_{ij}$, where $\alpha$ is a scaling factor and $\alpha \in \{0.1, 0.2, 0.3\}$. We tested the algorithm for a total of 60 instances. Tables \ref{tab:1} and \ref{tab:2} summarize the computational results of the algorithm. In Table \ref{tab:1}, the first column refers to the problem size, second column refers to the scaling factor $\alpha$, the third, fourth and the fifth columns refers to the cost of the optimal tour, cost of the ground vehicle tour and cost of the UAV tour averaged over five instances. In Table \ref{tab:2}, the average number of SEC cuts generated and the number of nodes explored in the branch-and-cut tree are listed in third and fourth columns. The last column indicates the number of instances out of five that were solved to optimality with in the allowed computation time of $9000$ seconds.

\renewcommand{\arraystretch}{0.75}
\begin{table}[htpb]
\centering
\caption{Average of optimal travel costs}
\begin{tabular}{ccccc}
\toprule
$|T|$ & $\alpha$ & Avg. opt. cost & Avg. GV cost & Avg. UAV cost \\
\midrule
 & 0.1 &	255.90 & 243.27 & 12.63 \\
10 & 0.2 & 266.74 & 246.89 & 19.85\\
 &	0.3 & 275.50 & 256.03 & 19.48 \\
 & & & & \\
 & 0.1 & 303.24 & 276.79 & 26.45 \\
20 & 0.2 & 328.68 & 280.08 & 48.60 \\
 & 0.3 & 351.19 & 290.80 & 60.40 \\
  & & & & \\
 & 0.1 & 283.53 &  246.85 & 36.68 \\
30 & 0.2 & 318.54 & 254.30 & 64.24 \\
 & 0.3 & 349.39 & 259.52 & 89.86 \\
  & & & & \\
 & 0.1 & 310.44 & 264.19 & 46.24 \\
40 & 0.2 & 357.79 & 268.86 & 88.94 \\
 & 0.3 & 397.91 & 281.71 & 116.19 \\
\bottomrule
\end{tabular}
\label{tab:1}
\end{table}

\begin{table}[htpb]
\centering
\caption{Branch-and-cut algorithm statistics}
\begin{tabular}{ccrrr}
\toprule
$|T|$ & $\alpha$ & SEC cuts & B{\&}C nodes & I \\
\midrule
& 0.1 & 167.00 & 5.80 & 5 \\
10 & 0.2 & 138.80 & 5.00 & 5 \\
& 0.3 & 170.80 & 6.20 & 5 \\
 & & & & \\
& 0.1 & 2206.60 & 46.00 & 5 \\
20 & 0.2 & 3213.20 & 54.20 & 5 \\
& 0.3 & 6263.40 &  113.20 & 5 \\
 & & & & \\
& 0.1 & 12764.80 & 247.60 & 5 \\
30 & 0.2 & 29796.40 & 576.80 & 5 \\
& 0.3 & 39399.40 & 761.40 & 5 \\
 & & & & \\
& 0.1 & 66680.60 & 1017.20 & 5 \\
40 & 0.2 & 102671.00 & 1372.80 & 2 \\
& 0.3 & 126649.80 & 1517.80 & 1 \\
\bottomrule
\end{tabular}
\label{tab:2}
\end{table}

The results tabulated in Tables \ref{tab:1} and \ref{tab:2} indicate that the proposed branch-and-cut algorithm can solve instances involving up to 40 targets with modest computation times. In summary we were able to solve 53/60 instances to optimality. For the remaining 7 instances, the algorithm produced feasible solutions that were within $3\%$ of the upper bound to the optimal (computed by CPLEX by solving corresponding dual problem). The plot in Table \ref{fig:plot} shows the average computation times for all the instances. 

\begin{table}[htpb]
\centering
\caption{Average computation time needed to find optimal solution}
\begin{tabular}{crrr}
\toprule
$|T|$ & $\alpha = 0.1$ & $\alpha = 0.2$ & $\alpha = 0.3$  \\
\midrule
10 & 0.98 & 0.89 & 0.89 \\
20 & 10.53 & 15.12 & 29.27 \\
30 & 220 & 665 & 853 \\
40 & 4588 & 7643 & 8156 \\
\bottomrule
\end{tabular}
\label{fig:plot}
\end{table}


With $\alpha$ equal to $0.1$, the maximum computation time needed to solve instances with $30$ targets is $420$ seconds, which is quite reasonable. In general, we observe that the instances with a scale factor of $\alpha = 0.3$ are more difficult to solve, and need more computation time. This is expected for the following reason: suppose the cost of travel by the UAV is zero, then the algorithm would try to assign as many targets as possible to the UAV. As the cost of travel by the UAV is increased, it has to find the right partitioning of the targets to be assigned to the ground vehicle and the UAV, that minimizes the total cost. This makes the problem combinatorially more difficult.

\section{Conclusion} \label{sec:conc}
We presented a path planning problem (CAGVRP) that arises in cooperative routing of a ground vehicle and an UAV with communication constraints. A MILP formulation is presented along with separation algorithms to find violating inequalities. The formulation and separation algorithms are integrated into a branch-and-cut framework, and implemented using CPLEX callable libraries. The algorithm was tested on several randomly generated instances with $10$, $20$, $30$ and $40$ targets. The algorithm could find optimal solutions for instances up to $40$ targets. The limitation of this algorithm is that it may not be able to solve larger instances and one might have to rely on heuristics. Future directions for this work includes finding better separation algorithms, addressing the nonlinear constraints in the branch-and-cut sub-routines instead of linearizing, and developing efficient heuristics. Also the problem can be generalized for multiple air vehicles and/or multiple ground vehicles.

\bibliographystyle{IEEEtran}
\bibliography{rcpacc.bib}

\end{document}